\documentclass[12pt]{article}
\usepackage[affil-it]{authblk}
\usepackage{amsmath} 
\usepackage{amssymb} 
\usepackage{makeidx} 
\usepackage{graphicx} 
\usepackage{enumerate}
\usepackage[utf8]{inputenc}
\usepackage{amsthm}
\usepackage{color}
\usepackage{cancel}

\theoremstyle{plain}
\newtheorem{theorem}{Theorem}[section]

\newtheorem{corollary}[theorem]{Corollary}

\theoremstyle{definition}

\newtheorem{remark}[theorem]{Remark}

\numberwithin{equation}{section}
\newtheorem{example}[theorem]{Example}

\long\def\symbolfootnote[#1]#2{\begingroup%
\def\thefootnote{\fnsymbol{footnote}}\footnote[#1]{#2}\endgroup}

\def\00{{\bf 0}}
\def\11{{\bf 1}}
\def\+{\oplus}

\newcommand{\boxtensor}{{\Box\kern-9.03pt\raise1.42pt\hbox{$\times$}}}

\newcommand{\F}{{\mathbb F}}

\newcommand{\cL}{\mathcal{L}}
\newcommand{\cO}{\mathcal{O}}

\newcommand{\Fq}{\mathbb{F}_q}
\newcommand{\Fp}{\mathbb{F}_p}

\newcommand{\be}{\begin{eqnarray}}
\newcommand{\ee}{\end{eqnarray}}

\begin{document}

\title{On the difference between permutation polynomials over finite fields}

\author{Nurdag\"{u}l Anbar$^1$, Almasa Od\u{z}ak$^2$, Vandita Patel$^3$, Luciane Quoos$^4$,  Anna Somoza$^{5,6}$, Alev Topuzo\u{g}lu$^7$
\vspace{0.4cm} \\
\small $^1$Otto-von-Guericke University Magdeburg, \\ 
\small Universit\"atsplatz 2, 39106 Magdeburg, Germany\\
\small Email: {\tt nurdagulanbar2@gmail.com}\\
\small $^2$University of Sarajevo,\\
\small Zmaja od Bosne 35, 71000 Sarajevo, Bosnia and Herzegovina\\
\small Email: {\tt almasa.odzak@gmail.com}\\
\small $^3$University of Warwick,\\
\small Coventry CV4 7AL, UK\\
\small Email: {\tt vandita.patel@warwick.ac.uk }\\
\small $^4$Universidade Federal do Rio de Janeiro, Cidade Universit\'{a}ria,\\
\small Rio de Janeiro, RJ 21941-909, Brazil\\
\small Email: {\tt luciane@im.ufrj.br}\\
\small $^5$Universitat Polit\`{e}cnica de Catalunya,\\
\small Calle Jordi Girona, 1-3, 08034 Barcelona, Spain\\
\small $^6$Leiden University ,\\
\small Snellius building, Niels Bohrweg 1 2300 RA Leiden, Netherlands\\
\small Email: {\tt anna.somoza@upc.edu }\\
\small $^7$Sabanc\i University,\\
\small MDBF, Orhanl\i, Tuzla, 34956 \. Istanbul, Turkey\\
\small Email: {\tt alev@sabanciuniv.edu}\\
 }

\date{}

\maketitle

\begin{abstract}
The well-known Chowla and Zassenhaus conjecture, proven by Cohen in 1990, states that  
if $p>(d^2-3d+4)^2$, then there is no complete mapping polynomial $f$ in 
$\Fp[x]$ of degree $d\ge 2$. 
For arbitrary finite fields $\Fq$,
a similar non-existence result is obtained recently by I\c s\i k, Topuzo\u glu and Winterhof
in terms of the Carlitz rank of $f$.

Cohen, Mullen and Shiue generalized the Chowla-Zassenhaus-Cohen Theorem significantly in 1995,
by considering differences of permutation polynomials. More precisely, they showed that if
$f$ and $f+g$ are both permutation polynomials of degree $d\ge 2$ over $\Fp$, with 
$p>(d^2-3d+4)^2$, then the degree $k$ of $g$ satisfies $k \geq 3d/5$, unless $g$ is constant. In this article, assuming
$f$ and $f+g$ are permutation polynomials in $\Fq[x]$, we give lower bounds for $k
$ in terms of the Carlitz rank of 
$f$ and $q$. Our results generalize the above mentioned result of I\c s\i k et al. 
We also show for a special class of polynomials $f$ of Carlitz rank $n \geq 1$ that if $f+x^k$ is a permutation of $\Fq$, with $\gcd(k+1, q-1)=1$, then $k\geq (q-n)/(n+3)$. 
\end{abstract}

\section{Introduction}

Let $\Fq$ be the finite field with $q=p^r$ elements, where $r\geq 1$ and $p$ is a prime. Throughout we assume $q \geq 3$.
We recall that $f \in \Fq[x]$ is a \textit{permutation polynomial} of $\Fq$ if it induces a bijection from $\Fq$ to $\Fq$. If $f(x)$ and $f(x)+x$ are both
permutation polynomials of $\Fq$, then $f$ is called a \textit{complete mapping}.
We refer the reader to \cite{NR} for a detailed study of complete mapping polynomials
over finite fields.
Their use in the construction of mutually orthogonal Latin squares is described, for instance, in \cite{LM98}.
For various other applications, see \cite{Muratovic2014, S00,SW10,SGCGM12}. The paper \cite{ITW} lists some recent work on complete mappings.

The Theorem 1 below was conjectured by Chowla and
Zassenhaus \cite{Chowla1968} in 1968, and proven by Cohen \cite{Cohen1990} in 1990. 

\textbf{Theorem 1.} 
If $d\geq2$ and $p>(d^2-3d+4)^2$, then there is no
complete mapping polynomial of degree $d$ over $\Fp$.

A significant generalization of this result was obtained by Cohen, Mullen and Shiue \cite{CMS} in 1995, and gives a lower bound for
the degree of the difference of two permutation polynomials in $\Fp[x]$ of the same degree $d$, when $p>(d^2-3d+4)^2$.

\textbf{Theorem 2.} 
Suppose $f$ and $f+g$ are monic permutation polynomials of $\Fp$ of degree $d \geq3$, where
$p>(d^2-3d+4)^2.$ If $\deg(g)=t\geq1$, then $t \geq 3d/5.$

An alternative invariant, the so-called Carlitz rank, attached to permutation polynomials, was used by I\c s\i k, Topuzo\u glu and Winterhof 
\cite {ITW} recently to obtain a non-existence result, similar to that in Theorem 1. The concept of Carlitz rank was first introduced in \cite {AAAW}. We describe it here briefly. The interested reader may see \cite{a} for details.

By a well-known result of Carlitz \cite{Carlitz} that any permutation polynomial of $\Fq$, with $q\ge 3$ is a composition of linear polynomials $ax+b$, $a,b \in \Fq$, $a\neq0$, and $x^{q-2}$, any permutation $f$ of $\Fq$ can be represented
by a polynomial of the form
\begin{equation}\label{pn}
P_{n}(x)=P_{n}(a_{0},...,a_{n+1};x)=\left(\ldots
\left(\left(a_0x+a_1\right)^{q-2}+a_2\right)^{q-2} \ldots +a_n\right)^{q-2}+a_{n+1},
\end{equation}
for some $n\geq 0$, where $a_i \neq 0$, for $i = 0,2,\ldots,n$. Note that $f(c)=P_{n}(c)$ holds for all $c \in \Fq$, however this representation is not unique,
and $n$ is not necessarily minimal. Accordingly 
the authors of \cite{AAAW} define the {\it Carlitz rank} of a 
permutation polynomial $f$ over $\Fq$ to be the smallest 
integer $n\ge 0$ satisfying $f=P_{n}$ for a permutation $P_{n}$ of the
form (\ref{pn}), and denote it by ${\rm Crk}(f)$.

The representation of $f$ as in (\ref {pn}) enables approximation of
$f$ by a fractional transformation in the following sense.

For $ 0 \leq k \leq n$, consider
\begin {equation}
\label{Rn}
{R}_k(x)= \frac{\alpha_{k+1}x+\beta_{k+1}}{\alpha_kx+\beta_k} \ ,
\end{equation}
where $\alpha_0=0, \alpha_1=a_0, \beta_0=1, \beta_1=a_1$ and, for $k \geq 2$, 
\begin{equation}\label{alpha}
\alpha_k = a_k\alpha_{k-1}+\alpha_{k-2}\hspace{.2in}\mbox{and}\hspace{.2in} \beta_k=a_k\beta _{k-1}+\beta_{k-2}\ .
\end{equation}
The set
\begin{equation}
\label{poles}
\cO_n=\left\{x_k:\;  x_k=\frac{-\beta_k}{\alpha_k}\ ,  \; k=1, \ldots,n \right\} \subset{ \mathbb{P}^1(\mathbb{F}_q)}=\mathbb{F}_q \cup \{ \infty \}\ ,
\end{equation}
is called the {\it set of poles} of $f$. The elements of $\cO_n$ may not be distinct. 

It can easily be verified that
\begin{equation}
\label{f=Rn} 
f(c)=P_{n}(c)=R_n(c)~~ {\rm for~ all ~~c} \in \Fq \setminus \cO_n \ . 
\end{equation}
Obviously, this property
is particularly useful when ${\rm Crk}(f)$ is small with respect to the field size.
The values that $f$ takes on $\cO_n$ can also
be expressed in terms of $R_n$, see \cite{a}. In case $\alpha_n=0$, i.e., the last \textit{pole} $x_n=\infty$, $R_n$ is linear. Following the terminology of \cite{ITW}, we define the \textit{linearity} of $f \in \F_q[x]$ as  
${\cal L}(f)=\max_{a,b\in \F_q} |\{ c\in \F_q : f(c)=ac+b\}|.$ Intuitively ${\cal L}(f)$ is large when $f$ is a permutation polynomial of $\Fq$ of ${\rm Crk}(f)=n, ~R_n$ is linear,
and $n$ is small with respect to $q$.

Now we are ready to state the main result of \cite{ITW}. We remark that the Theorems 1 and 2 hold over prime fields only, while the Theorem 3 is true for any finite field.

\textbf{Theorem 3.} If $f(x)$ is a complete mapping of $\F_q$ and 
${\cal L}(f)< \left\lfloor (q+5)/2\right\rfloor,$
then
${\rm Crk}(f) \ge \left\lfloor q/2\right\rfloor.$

The purpose of this note is to obtain a lower bound for the degree of the difference between two permutation polynomials, analogous to Theorem 2, generalizing Theorem 3. In what follows we assume that $f$ and $f+g$ are permutation polynomials of $\Fq$, where $g \in \Fq[x]$ has degree $k$ with $1 \leq k<q-1$.
We give lower bounds for $k$ in terms of $q$ and the Carlitz rank of $f$, see Theorems \ref{thm:main} and \ref{thm:monomial} below.

\section{Degree of the difference of two permutation polynomials}

Let $f$ be a permutation polynomial of $\Fq,~q \geq 3$, with ${\rm Crk}(f)=n\geq 1$.
Suppose that $f$ has a representation as in (\ref{pn}) and 
the fractional linear transformation $R_n$ in (\ref{Rn}), which is associated to $f$ 
as in (\ref{f=Rn}) is not linear, in other words $\alpha_n$ in (\ref{Rn}) defined as in (\ref{alpha}) is not zero. 
We denote the set of all such permutations by $\cL_1$, i.e., the set $\cL_1$ consists of all permutation polynomials of $\mathbb{F}_q$, satisfying ${\rm Crk}(f)=n\geq 1$ and $\alpha_n \neq 0$. Clearly $\cL(f) \leq n+2$, if $f \in \cL_1$.
We note that permutations $f\in \mathbb{F}_q[x]$ with $\alpha_n=0$ behave very differently. For instance, there are examples of complete mappings of $\mathbb{F}_q$ of Carlitz rank $4$ for infinitely many values of $q$. Indeed, the condition on the linearity of $f$ in Theorem $3$ corresponds to the case $\alpha_n=0$.
Therefore, we only consider permutations in $\mathcal{L}_1$.

We now prove our main theorem. 
\begin{theorem}\label{thm:main}
Let $f$ and $f+g$ be permutation polynomials of $\mathbb{F}_q$, where $\mathrm{Crk}(f)=n \geq 1$, $f \in \cL_1$, and the degree $k$ of $g \in \F_q[x] $ satisfies $1 \leq k<q-1$. Then 
\begin{equation} \label{eq:main}
nk+k(k-1)\sqrt{q} \geq q-\nu-n \ , \
\end{equation}
where $\nu=\mathrm{gcd}(k, q-1)$.  
\end{theorem}
\begin{proof}

Since $f \in \cL_1$, there exist $a,b, d\in \mathbb{F}_q $, such that $f(z)=R_n(z)$
for $ z \in \mathbb{F}_q \setminus \mathcal{O}_n$, where
\begin{equation*}
R_n(z)=\frac{az+b}{z+d} \ . 
\end{equation*}
The fact that $ad-b\neq 0$ follows from (\ref{alpha}).

The polynomial $f(z)+g(z)$ can be represented by $G_n(z)= R_n(z)+g(z)$ for $ z \in \mathbb{F}_q \setminus \mathcal{O}_n$.
Since $f+g$ is a permutation of $\Fq$, the map $G_n$ is injective on $\mathbb{F}_q \setminus \mathcal{O}_n$. 

For $u\in \mathbb{F}_q$ and
\begin{equation}\label{eq=perm}
G_n(z)=\frac{az+b}{z+d}+g(z)=u \ , \ 
\end{equation}
we set
\begin{equation*}
H_n(x):=G_n(x-d)=\frac{ax-\tilde{b}}{x}+h(x)=u \ . \ 
\end{equation*}
where $\tilde{b}=ad-b$ and $h(x)=g(x-d)$.
Note that $H_n(x)=u$ for some nonzero $x\in \mathbb{F}_q$ if and only if
$z \neq -d$ is a solution of Equation \eqref{eq=perm}.
Let $S$ be the set of pairs $(x,y) \in \Fq^* \times \Fq^*$ such that 
\begin{align*}
S=\left\lbrace \, (x,y) \, : \; x\neq y \;\; \text{and} \;\; H_n(x)=H_n(y)\; \right\rbrace \ .
\end{align*}
Denote the value set of $H_n$ by $V_{H_n}$, i.e.,
$$V_{H_n}= \{u \in \Fq: \exists \, x \in \Fq \;\; \text{with} \;\; H_n(x)=u \}\ .$$ 
Suppose that the cardinality $|S|$ of $S$ is $\mu$.
For $u\in V_{H_n}$, we set $H_n^{-1}(u)=\{x \in \Fq: \; H_n(x)=u \}$ and put $n_u=|H_n^{-1}(u)|$.
We remark that $0\not\in H_n^{-1}(u)$ and that $x\in H_n^{-1}(u)$ if and only if $x$ is a root of the polynomial
\begin{align}\label{eq:xh(x)}
xh(x)+(a-u)x-\tilde{b} \ .
\end{align}
This shows that for any $u\in V_{H_n}$ we have $n_u\leq k+1$ as the polynomial in Equation \eqref{eq:xh(x)} has degree $k+1$.
We then conclude that 
\begin{align}\label{eq:mu}
\mu=\sum_{u\in V_{H_n}} n_u(n_u-1) \leq (k+1)\sum_{u\in V_{H_n}} (n_u-1) \ .
\end{align}
If there exist $n_u$ distinct elements $x$ with $H_n(x)=u$, then there exist $n_u$ distinct elements $z$ with $G_n(z)=u$. 
Since $G_n(z)$ is injective on $\mathbb{F}_q \setminus \mathcal{O}_n$, this shows that $n_u-1$ distinct elements $z$ lie in the set of poles $O_n$. In particular, by Equation \eqref{eq:mu} and the fact that $-d\in O_n$ we conclude that
\begin{align*}
n\geq |O_n| \geq 1+ \sum_{u\in V_{H_n}} (n_u-1) \geq 1+\frac{\mu}{k+1} \ . 
\end{align*}
Therefore in order to obtain a lower bound for $k$ in terms of $q$ and $n$, it is sufficient to determine $\mu$ in relation to $q$ and $k$.

We can re-write the equation $H_n(x)=H_n(y)$ as 
\begin{equation*} 
y(xh(x)-\tilde{b})-x(yh(y)-\tilde{b})=0 \ .
\end{equation*}
Note that $x-y$ is a factor of $y(xh(x)-\tilde{b})-x(yh(y)-\tilde{b})$.
We want to find an absolutely irreducible factor over $\mathbb{F}_q$ of the polynomial in two variables of degree $k+1$ defined by
\begin{equation*}\label{eq:new1}
\frac{y(xh(x)-\tilde{b})-x(yh(y)-\tilde{b})}{x-y} \ ,
\end{equation*}
or equivalently defined by
\begin{equation}\label{eq:new12}
xy\frac{h(x)-h(y)}{x-y}+\tilde{b} \ .
\end{equation}
  
We recall that a rational function $\ell(x)/t(x)\in \mathbb{F}_q(x)$ is called \textit{exceptional} over $\mathbb{F}_q$ if the polynomial $\Theta_{\ell/t}$, defined by
\begin{equation*}
\Theta_{\ell/t}:=\frac{t(Y)\ell(X)-t(X)\ell(Y)}{X-Y} 
\end{equation*}
has no absolutely irreducible factor in $\mathbb{F}_q[X,Y]$. 
By Theorem 5 of 
\cite{C1970}, $\ell/t$ is a permutation of $\mathbb{F}_q$ if it is an exceptional function over $\mathbb{F}_q$.
In particular, $t(\alpha)\neq 0$ for all $\alpha \in \mathbb{F}_q$. Now we put $\ell/t=(xh(x)-\tilde{b})/x$, and
conclude that the rational function in Equation \eqref{eq:new12} has an absolutely irreducible factor $p(x,y)$ over $\mathbb{F}_q$.

Consider the curve $\mathcal{X}$ whose affine equation is given by $p(x,y)$ of degree $\varrho \leq k+1$. Then by \cite[Theorem 9.57]{HKT2013} we conclude that the number of rational points $N(\mathcal{X})$ in $PG(2, q)$ of $\mathcal{X}$ is bounded by
\begin{equation*}
N(\mathcal{X})\geq q+1-(\varrho-1)(\varrho-2)\sqrt{q}\geq q+1-k(k-1)\sqrt{q} \ .
\end{equation*}

We denote by $P(X,Y,Z)$ the homogenized polynomial of $p(x,y)$, i.e., 
\begin{equation*}
P(X,Y,Z)=Z^{\varrho}p\left( \frac{X}{Z}, \frac{Y}{Z}\right) \ .
\end{equation*}
In order to find the number of affine solutions $(x :y: 1)$ such that $xy\neq 0$ and $x\ne y$, we proceed as follows. 
From Equation \eqref{eq:new12} we have that $P(X,Y,Z)$ is a divisor of the homogeneous polynomial
\begin{equation}\label{eq:homo}
XYZ^{k-1} \left( \frac{h(X/Z)-h(Y/Z)}{X-Y} \right)+\tilde bZ^{k+1} \ .
\end{equation}
Hence we conclude that there is no affine solution $(x :y:1)$ of $P(X,Y,Z)$ with $xy=0$. 
We now estimate the number of rational points of $\mathcal{X}$ at infinity, i.e., the points of the form $(x :y:0)$ for $x,y \in \mathbb{F}_q$. By Equation \eqref{eq:homo} the point $(x :y:0)$ is on $\mathcal{X}$ only if
\begin{equation*}
xy \frac{x^k-y^k}{x-y}=0 \ .
\end{equation*} 
This holds only if $(x :y:0)=(0 :1:0),(1 :0:0)$ or $x^k=y^k$ for some $x,y \in \mathbb{F}_q^*$. Since $\nu=\gcd (k, q-1)$, the equality $x^k=y^k$ is satisfied if and only if $x/y$ is an $\nu$-th root of unity in $\mathbb{F}_q$. Hence there exist at most $\nu+2$ rational points of $\mathcal{X}$ lying at infinity. 

Bezout's theorem implies that there are at most $k+1$ rational points $(x:y:z)$ of $\mathcal{X}$ with $x=y$, 
since the degree of $\mathcal{X}$ is at most $k+1$.

This shows that the cardinality $\mu$ of the set $S$ satisfies
\begin{equation*}
\mu  \geq q+1-k(k-1)\sqrt{q}-(\nu+k+2)\ .
\end{equation*}
Note that we subtract $\nu+k+2$ instead of $\nu+k+3$. This is because of the point $(1:1:0)$. If $(1:1:0)$ is on $\mathcal{X}$ then it is taken into account twice. If it is not on $\mathcal{X}$ then we do not have to exclude it as a point at infinity. Therefore $\mathrm{Crk}(f)=n$ satisfies 
\begin{align*}
n&\geq 1+\frac{1}{k+1}(q+1-k(k-1)\sqrt{q}-(\nu+k+2))\\
& =\frac{1}{k+1}(q-k(k-1)\sqrt{q}-\nu) \ ,
\end{align*}
which implies the desired result.
\end{proof}

For $k=1$ (and hence $\nu=1$) we obtain Theorem 3, i.e., the main result in \cite{ITW}.
\begin{corollary}\label{casegen}
Let $f$ be a permutation polynomial of $\mathbb{F}_q$, with $\mathrm{Crk}(f)=n \geq 1$ and $f \in \cL_1$. If $n<(q-1)/2$, then $f$ is not a complete mapping.
\end{corollary}
\begin{remark}
We note that the bound given in \eqref{eq:main} is non-trivial only when $q\geq k(k-1)\sqrt{q}+k+\nu+1 $.
\end{remark}

\section{The case $g(x)=cx^k$}

Throughout this section we focus on monomials $g(x)=cx^k\in \Fq[x]$ and $f\in \cL_1$, where $x_n\in \cO_n$ in \eqref{poles} satisfies $x_n=0$. In this particular case, the lower bound in \eqref{eq:main} can be simplified significantly when $\gcd (k+1,q-1)=1$. Let $\cL_2$ be the set of $f\in \cL_1$ such that the last pole $x_n$ of $f$ is zero.

\begin{theorem}\label{thm:monomial}
Let $f(x)$ and $f(x)+ cx^k$ be permutation polynomials of $\mathbb{F}_q$, where $\mathrm{Crk}(f)=n \geq 1,
~f \in \cL_2,~1 \leq k < q-1 , ~c \in \Fq^*$. Put $m=\gcd (k+1,q-1)$. Then
\begin{equation*}
k(n+3)+(k-1)(m-1)\sqrt{q}\geq q-n \ .
\end{equation*}
In particular, if $m=1$, then $k\geq (q-n)/(n+3)$.
\end{theorem}
\begin{proof}
The condition $x_n=0$ implies that $\beta_n$ in \eqref{alpha} is zero. Hence we have $R_n(x)=\frac{ax+b}{x}$ for some $a,b\in \mathbb{F}_q$, with $b\neq 0$. That is, for $x \in \mathbb{F}_q\setminus \mathcal{O}_n $ we can represent $f+cx^k$ by $G_n(x)=R_n(x)+cx^k$. 

We proceed as in the proof of Theorem 2.1. 
The equation $G_n(x)=u$ for some $u\in \mathbb{F}_q$ becomes 
\begin{equation*}
\frac{ax+b}{x}+cx^k=u \ . 
\end{equation*} 
Then for some nonzero $x,y \in \mathbb{F}_q$, we have $G_n(x)=G_n(y)$ if and only if the equation
\begin{equation*}
cx^k+\frac{b}{x}=cy^k+\frac{b}{y} \ ,
\end{equation*}
or equivalently the equation
\begin{equation}\label{eq:curve1}
x^k-y^k=\frac{b}{c}\left( \frac{x-y}{xy}\right) 
\end{equation}
holds.

We again consider the set $S$ of pairs $(x,y) \in \Fq^* \times \Fq^*, ~x \neq y$, where $(x,y)$ is a solution of (\ref{eq:curve1}),
and denote the cardinality of $S$ by $\mu$. By using the argument given in the proof of Theorem 2.1, we conclude that $n \geq 1+\mu/(k+1)$. Hence our aim now is to express $\mu$ in terms of $q$
and $k$. 

Applying the change of variable $(x,y) \rightarrow (xy,y)$, Equation \eqref{eq:curve1} becomes 
\begin{equation*}
y^k(x^k-1)=\frac{b(x-1)}{cxy} \ .
\end{equation*}
Hence we are looking for the affine points $(x,y) \in \Fq^* \times \Fq^*$ of the curve
\begin{equation}\label{eq:curve2}
\mathcal{X}: \; y^{k+1}=\frac{b(x-1)}{cx(x^k-1)} \ .
\end{equation}
Note that in this case the solutions should not lie in the set $\lbrace(\gamma^2,\gamma)\, |\, \gamma\in \mathbb{F}_q \rbrace$.
Recall that $m=\gcd (k+1,q-1)$, hence the monomial $y^{(k+1)/m}$ gives rise to a permutation over $\Fq^*$. Hence there is one-to-one correspondence between the affine solutions in $(x,y) \in \Fq^* \times \Fq^*$ of the curve 
\begin{equation}\label{eq:curve3}
\mathcal{Y}: \; y^{m}=\frac{b(x-1)}{cx(x^k-1)} 
\end{equation}
and $\mathcal{X}$ given in Equation \eqref{eq:curve2}. Equation \eqref{eq:curve3} defines a Kummer extension. Then by using arithmetic of function fields, see \cite[Proposition 3.7.3 ]{sti}, we can estimate the number of $\mathbb{F}_q$-rational points of $\mathcal{Y}$ as follows.

For the rational function field $\mathbb{F}_q(x)$ and  $\alpha \in \mathbb{F}_q$, we denote by $(x=\alpha)$ and $(x=\infty)$ the places corresponding to the zero and the pole of $x-\alpha$, respectively. Let $F=\mathbb{F}_q(x,y)$ be the function field of $\mathcal{Y}$ defined by Equation \eqref{eq:curve3}. Then the ramified places of $\mathbb{F}_q(x)$ in $F$ are exactly 
\begin{equation*}
(x=0), \quad (x=\infty) \quad \text{and} \quad (x=\alpha) \quad \text{with} \quad \alpha^k=1 \; \text{and} \; \alpha \neq 1 \ .
\end{equation*}
It is clear that the places $(x=0)$ and $(x=\alpha)$ are totally ramified in $F$. In particular, this shows that the full constant field of $F$ is $\mathbb{F}_q$. For the place $(x=\infty)$ we have the ramification index $e_{\infty}=m/\gcd (m,k)=m$, since $m$ is a divisor of $k+1$. We conclude that the degree of the different divisor of $F/\mathbb{F}_q(x)$ is $(k+1)(m-1)$. Then by the Hurwitz genus formula the genus $g(F)$ of $F$ satisfies
\begin{equation*}
2g(F)-2=-2m+(k+1)(m-1) \ ,
\end{equation*}
which implies that $g(F)=(k-1)(m-1)/2$. By the Hasse--Weil theorem the number $N(F)$ of $\mathbb{F}_q$-rational places of $F$ is bounded by
\begin{equation}\label{eq:numberY}
N(F)\geq q+1-2g(F)\sqrt{q}=q+1-(k-1)(m-1)\sqrt{q} \ .
\end{equation}
We observe that the pole divisors $(x)_{\infty}$, $(y)_{\infty}$ of $x, y$ are 
\begin{align*}
(x)_{\infty}= mP_{\infty} \quad \text{and}\quad (y)_{\infty}= P_0+\sum_{\alpha^k=1  ,\alpha \neq 1 }P_{\alpha}\ ,
\end{align*}
where $P_{\infty}, P_0,P_{\alpha}$ are the unique places of $F$ lying over $(x=\infty),(x=0),(x=\alpha)$, respectively.

We remark that curve $\mathcal{Y}$ defined by Equation \eqref{eq:curve3} is of degree $k+m$ and has two points at infinity; namely $Q_1=(1:0:0)$ and $Q_2=(0:1:0)$. These are the only singular points of $\mathcal{Y}$ and $Q_1$ has intersection multiplicity $m$ while $Q_2$ is a $k$ ordinary multiple point. Moreover, $P_{\infty}$ is the unique place corresponding to $Q_1$, and there are $k$ places corresponding to $Q_2$, which correspond to the places lying in the support of $(y)_{\infty}$.
All the affine points in the curve $\mathcal{Y}$ defined by Equation \eqref{eq:curve3} are non singular and  there is a one to one correspondence between these points and the places in the function field $F$ of $\mathcal{Y}$ which does not lie in the support of pole divisors of $x$ and $y$. Moreover, the fact that the zero divisors of $x$ and $y$ are $(x)_{0}= mP_{0}$ and $(y)_{0}= kP_{\infty}$, respectively, implies that the rational places not lying in the pole divisors correspond to points $(x,y)\in \Fq^* \times \Fq^*$. Therefore we conclude, by Equation \eqref{eq:numberY}, the number of affine points $(x,y)\in \Fq^* \times \Fq^*$ of $\mathcal{Y}$ is at least $q-(k-1)(m-1)\sqrt{q}-k$.

Now we return the curve $\mathcal{X}$ given by Equation \eqref{eq:curve2}. We have seen that $\mathcal{X}$ has at least $q-(k-1)(m-1)\sqrt{q}-k$ many affine points $(x,y) \in \Fq^* \times \Fq^*$. Next we estimate the number of affine points $(x,y)$ of $\mathcal{X}$ such that $(x,y)$ is not of the form $(\gamma^2,\gamma)$ for some $\gamma \in \mathbb{F}_q$. By Equation \eqref{eq:curve2}, the affine point $(\gamma^2,\gamma)$ lies on $\mathcal{X}$ if and only if $\gamma$ is a root of
\begin{equation}\label{eq:T}
T^{k+1}\sum_{i=1}^{k}T^{2i}-\frac{b}{c} \ .
\end{equation}
Since polynomial in Equation \eqref{eq:T} has degree $3k+1$, there can be at most $3k+1$ such many points. 
Hence the number $\mu$ of affine solutions $(x,y) \in \Fq^* \times \Fq^*$ of Equation \eqref{eq:curve2}, which do not lie on the curve $x=y^2$ satisfies
\begin{align*}
\mu \geq q-(k-1)(m-1)\sqrt{q}-(4k+1)  \ .
\end{align*}
Therefore $\mathrm{Crk}(f)=n$ satisfies
\begin{equation*}
n \geq 1+\frac{1}{k+1}( q-(k-1)(m-1)\sqrt{q}-(4k+1)) \ ,
\end{equation*}
which gives the desired conclusion.
\end{proof}

\begin{example}\label{ex}
For $q=9$, $n=3$ and $m=1$, the bound in Theorem \ref{thm:monomial} gives $k\geq 1$. Combining with Corollary \ref{casegen} we get $k\geq 2$ as $q >2n+1$.
Let $\zeta$ be a primitive element of $\F_9$ and $f(x)=(((x+a)^7)+b)^7+c)^7$ be the permutation polynomial of Carlitz rank $3$, where $a=\zeta^5$, $b=\zeta^6$ and $c=\zeta^3$ over $\F_9$. It can be checked easily that the polynomial $f(x)+x^2$ is a permutation over $\mathbb{F}_9$.
\end{example}

\begin{remark}\label{direct} As we have seen in Example \ref{ex}, the bound in Theorem \ref{thm:monomial} is weaker than the one in Theorem \ref{thm:main} for $k=1$. The reason is the change of variable $(x,y) \rightarrow (xy,y)$ in the proof of Theorem \ref{thm:monomial}. However, a direct calculation in this specific case is possible, and gives an alternative proof for Theorem 3, which was proven in \cite{ITW}. In fact, the change of variable is not needed when $k=1$ as Equation \eqref{eq:curve1} becomes $xy=b$. In this case, each non-zero $x$ uniquely determines $y$, i.e., there exists $q-1$ distinct solutions $(x,y)$ of $xy=b$. We also leave out the solutions $(x,y)$ with $x=y$. We therefore obtain $\mu=q-2$ if $q$ is even, and $\mu=q-3$ or $q-1$ (depending on $b$ being square or not) if $q$ is odd. Then the fact that $n\geq 1+\mu/2$ implies Corollary 2.2.
\end{remark}

\section*{Acknowledgement}

The initial work on this project began during ``Women in Numbers Europe 2 (WIN-E2)" workshop, held in Lorentz Center, Leiden in September 2016.  The authors are grateful to Lorentz Center and all supporting institutions for making this conference and collaboration possible. They would especially like to thank the organisers of WIN-E2, Irene Bouw, Rachel Newton and  Ekin \"{O}zman for all of their hard work, as this resulted in an extremely fruitful and enjoyable meeting.

The authors N.A.; A.O.; V.P.; L.Q. and A.T. are partially supported by H.C.~\O rsted COFUND Post-doc Fellowship 
from the project ``Algebraic curves with many rational points"; Federal Ministry of Education and Science, 
grant No.05-39-3663-1/14; an EPSRC studentship; CNPq, PDE grant number  200434/2015-2 and TUB\.ITAK project 
number 114F432, respectively.

\end{document}